\documentclass[11pt]{amsart}
\usepackage{enumitem}
\usepackage{amsmath}
\usepackage{amssymb, amsmath, amscd, amsthm}
\usepackage{graphicx}
\usepackage{xcolor}
\usepackage[margin=1in]{geometry}
\usepackage[hidelinks]{hyperref}
\usepackage[final]{showlabels}

\def\div{\operatorname{div}}
\DeclareMathOperator{\sgn}{sgn}

\newcommand{\defeq}{\mathrel{\mathop:}=}
\newcommand{\eqdef}{\mathrel{\mathop=}:}
\newcommand{\eval}[2]{\langle #1,#2 \rangle}
\newcommand{\Eval}[2]{\bigl\langle #1,#2 \bigr\rangle}

\newcommand{\simplex}[2][0]{[#1\,...\,#2]}

\def\R{\mathbb{R}}

\newcommand{\somethingdef}[4]{{\left#1 {#2} \ \left| \ {#3} \vphantom{#2} \right. \right#4}}
\newcommand{\setdef}[2]{\somethingdef{\{}{#1}{#2}{\}}}

\newtheorem{theorem}{Theorem}[section]
\newtheorem{proposition}[theorem]{Proposition}

\newtheorem{lemma}[theorem]{Lemma}
\theoremstyle{definition}
\newtheorem{definition}[theorem]{Definition}
\newtheorem{example}[theorem]{Example}
\theoremstyle{remark}

\newtheorem{remark}[theorem]{Remark}
\newtheorem{notation}[theorem]{Notation}

\graphicspath{{figs/}}

\begin{document}
\title{Averaging Property of Wedge Product and Naturality in Discrete Exterior Calculus}
\author{Mark D. Schubel}
\email{mdschubel@gmail.com}
\address{Department of Physics, University of Illinois at Urbana-Champaign, 1110 West Green Street, Urbana, IL 61801}
\curraddr{Apple Inc., One Apple Park Way, Cupertino, CA 95014}
\author{Daniel Berwick-Evans}
\email{danbe@illinois.edu}
\address{Department of Mathematics, University of Illinois at Urbana-Champaign, 1409 West Green Street, Urbana, IL 61801}
\author{Anil N. Hirani*}\thanks{*Corresponding author, hirani@illinois.edu}
\email{hirani@illinois.edu}
\address{Department of Mathematics, University of Illinois at Urbana-Champaign, 1409 West Green Street, Urbana, IL 61801}
\thanks{It is a pleasure to dedicate this paper to Alain Bossavit on his 80th birthday. Alain's writings have been very influential to ANH. When ANH was a graduate student he had the honor of inviting Alain to a workshop at Caltech and has wonderful memories of the many discussions with Alain.}

\begin{abstract}
In exterior calculus on smooth manifolds, the exterior derivative and wedge product are natural with respect to smooth maps between manifolds, that is, these operations commute with pullback. In discrete exterior calculus (DEC), simplicial cochains play the role of discrete forms, the coboundary operator serves as the discrete exterior derivative, and the antisymmetrized cup product provides a discrete wedge product. We show that these discrete operations in DEC are natural with respect to abstract simplicial maps. A second contribution is a new averaging interpretation of the discrete wedge product in DEC. We also show that this wedge product is the same as Wilson's cochain product defined using Whitney and de Rham maps.

\end{abstract}

\maketitle

\section{Introduction}
Discrete exterior calculus (DEC) builds a combinatorial version of exterior calculus on smooth manifolds with manifolds replaced by cell complexes (usually simplicial complexes)~\cite{Hirani2003, DeHiLeMa2005}. This is done by replacing the objects and operators of exterior calculus by discrete ones in a way that faithfully encodes expected algebraic identities, e.g., the Leibniz rule for the de~Rham differential. The Riemannian metric is encoded in DEC via a primal and dual cell complex that incorporate orthogonality, lengths, areas, volumes etc. 
There are other such frameworks with similar goals, with a prominent one being finite element exterior calculus (FEEC)~\cite{ArFaWi2010} in which differential forms are replaced by piecewise polynomial forms with certain continuity properties.

Like FEEC, for most of its existence DEC has been viewed as a framework for numerically solving partial differential equations (PDEs) on cell complexes. For example, FEEC, DEC and their antecedents have been used in computational electromagnetism~\cite{Bossavit1988a,Hiptmair1999}, elasticity~\cite{ArWi2002,ArFaWi2007, Li2018}, numerical relativity~\cite{Quenneville-Belair2015,Li2018}, fluid mechanics~\cite{HiNaCh2015,MoHiSa2016,PaGe2017,NiTeVo2017,JaAbMoSa2021,WaJaHiSa2023}, quantum electrodynamics~\cite{PhFaScTu2023} and many other areas of physics and geometry~\cite{ArHu2021}. 

Algebraic structures on simplicial cochains have also been studied for more theoretical purposes. In particular, the simplicial coboundary operator and discrete wedge product can be used to obtain combinatorial manifestations of various topological and geometric invariants e.g., see~\cite{Kervaire1953, Whitney1957, Sullivan1977, RaSu1976, Wilson2007, Dupont1978}. See~\cite{Wilson2007} for the context for these and other related references.

In this paper we promote a category-theoretic viewpoint for DEC that generalizes the situation in smooth geometry. Smooth exterior calculus takes place within the category whose objects are manifolds and morphisms are smooth maps between manifolds. Naturality of exterior calculus (in the category-theoretic sense) encodes fundamental structures, e.g., the chain rule. In DEC the objects are clearly simplicial complexes, but the appropriate morphisms analogous to smooth maps have not been spelled out previously. In the absence of topological and differentiable structure on the simplicial complexes (as is the case in DEC) it is not a priori clear what such morphisms should be. We propose abstract simplicial maps as discrete proxy for smooth maps by showing that discrete exterior derivative and wedge product commute with pullback by abstract simplicial maps. This mimics the analogous properties in the smooth case, including a discrete chain rule.

Roughly speaking, wedge products are required for PDEs with nonlinear terms. For example, nonlinearity may present itself as a product of functions, such as in the term $\div(\phi u)$ in two phase flow with $\phi$ indicating the phase and $u$ the velocity~\cite{WaJaHiSa2023}. Or it may arise indirectly, for example in the convective term $u\cdot \nabla u$ in Navier-Stokes equations for incompressible flow which leads to a wedge product via a Lie derivative. That is, $u\cdot \nabla u$ is the vector proxy for the form
$
\mathcal{L}_u u^\flat - \frac{1}2 d i_u u^\flat = i_u du^\flat + \frac{1}2 d i_u u^\flat\, ,
$
and contraction can be written in terms of wedge product, $i_X\alpha = \pm *(*\alpha\wedge X^\flat)$ for a form $\alpha$ and vector field $X$. The sign depends on the differential form degree. This is the formulation that was used in~\cite{MoHiSa2016}.

Piecewise polynomial approximations of wedge product have not yet appeared in the FEEC literature, except in the special case of wedge products used to define inner products of forms. A FEEC theory incorporating general wedge products will likely need to address the fact that degrees add under products of polynomials. Thus the finite element space of the product $\alpha \wedge \beta$ is different from that of the constituent factors $\alpha$ and $\beta$. One approach might be to convert a polynomial form to degrees of freedom (DOF) values followed by a combinatorial wedge product on the DOF. This can then be mapped to a shape function value. The discrete DEC wedge product is such a combinatorial product operation that might be useful in such constructions at least for the lowest degree Whitney forms used in FEEC. In fact, the DEC wedge product is closely related to the wedge product of Whitney forms as we show. Within DEC the old combinatorial product operator from~\cite{Hirani2003} has been recently used in the discretization of nonlinear terms in fluid mechanics PDEs~\cite{MoHiSa2016, WaJaHiSa2023}.

A second contribution of this paper is a satisfying interpretation of this old DEC discrete wedge product. This interpretation organizes the (complicated) combinatorics of the wedge product into conceptually simpler averages and products of values of cochains involved. With the coboundary operator interpreted as a difference operator and the averaging interpretation of wedge product, the metric independent parts of DEC are seen to be consisting of simple arithmetic operations, making DEC a useful tool that requires minimal mathematical background as compared with exterior calculus on differentiable manifolds.

\subsection*{Statement of results}

Propositions~\ref{prop:dnaturality} and~\ref{prop:wedge_naturality} prove the naturality property of the discrete exterior derivative and wedge product under pullback by abstract simplicial maps. Proposition~\ref{prop:averaging} gives a new interpretation of anti-symmetrized cup product which was defined in DEC in~\cite{Hirani2003} and which has been used in some physical applications. This new interpretation is a double averaging involving the two cochains involved. Proposition~\ref{prop:wilson} shows that the DEC wedge product is equal to a cochain product of Wilson~\cite{Wilson2007} and hence the averaging interpretation applies to Wilson's cochain product as well.

\section{Background: Discrete exterior calculus}\label{sec:DEC}
In this section we give a brief overview of DEC~\cite{Hirani2003}.  The input data is a simplicial complex $X$ with additional decorations and properties. The discrete notions of differential form, exterior derivative, and wedge product only depend on the underlying simplicial complex, and are defined using standard methods from simplicial algebraic topology. Incorporating features that depend on a metric (e.g., a discretization of the Hodge star operator) essentially requires that $X$ approximates a manifold. This assumption is appropriate given the desired applications: DEC has been used mostly as a method for solving PDEs on simplicial approximations of embedded orientable manifolds.  A discrete Hodge star construction involves a Poincar\'e dual complex of $X$ using circumcenters~\cite{Hirani2003} and is not relevant to this paper.

With the above in mind, below we will assume that $X$ arises as an approximation of an embedded manifold. In particular, each top dimensional simplex is embedded in $\mathbb{R}^N$ individually, and combinatorial data specifies how these are glued to each other. This may be presented by embedding the entire approximation of the manifold as a complex of dimension $m$ embedded in $\mathbb{R}^N$ for some $N \ge m$. A common example is a piecewise-linear (PL) approximation of a surface in $\mathbb{R}^3$. But the coordinate-independent aspect of DEC does not require such a global embedding. All the operations and objects are local to the simplices and their neighbors. In DEC, the top dimensional simplices of the simplicial approximation of an orientable manifold are oriented consistently and the lower dimensional simplices are oriented arbitrarily.

Thus the starting point for DEC is a simplicial complex $X$ which may be a triangulation approximation of $M$. For the results of this paper $X$ can be simply a triangulation of a manifold. Then let $C_k(X)$ denote the vector space of $k$-chains defined over $\R$, and $C^k(X)$ the corresponding space of $k$-cochains. Given a differential form $\alpha \in \Lambda^k(M)$, one obtains a $k$-cochain $\int_{\textvisiblespace }\alpha$ via integration over $k$-chains, i.e., the value of the de~Rham map~\cite{Dodziuk1976}. By Stokes theorem, the coboundary operator on cochains plays the role of discrete exterior derivative, denoted by~$d$ below. From standard algebraic manipulations, the cup product ($\smile$) plays the role of tensor product and the antisymmetrized cup product plays the role of a discrete wedge product ($\wedge$). For $\alpha \in C^k(X)$ the notation $\eval{\alpha}{[v_0,\ldots,v_k]}$ denotes evaluation of $\alpha$ on the oriented simplex $[v_0,\ldots, v_k]$. Often we will use $[0\ldots k]$ to label a generic oriented $k$-simplex in $X$. The following definitions of the discrete exterior derivative and discrete wedge product in DEC are from~\cite{Hirani2003}. Each is defined below on a simplex and extends by linearity to chains.

\begin{definition}[DEC exterior derivative]\label{defn:d}
    For a cochain $\alpha \in C^k(X)$, the \emph{discrete exterior derivative} $d\alpha\in C^{k+1}(X)$ is characterized by its evaluation $(k+1)$-dimensional simplices $\sigma$ as
    $\eval{d\alpha}{\sigma} := \eval{\alpha}{\partial \sigma}$ and extending by linearity to $C^{k+1}(X)$. Recall that if $\sigma =[0\ldots k+1]$ the boundary $\partial \sigma = \sum_{i=0}^{i=k+1} (-1)^i [0\ldots \widehat{i}\ldots k+1]$ where $\widehat{i}$ means missing vertex.
\end{definition}
\begin{definition}[DEC wedge product]\label{defn:wedge}
  Given cochains $\alpha \in C^k(X)$ and $\beta \in C^l(X)$ the \emph{wedge product} $\alpha \wedge \beta$ is characterized by its evaluation on $(k+l)$-simplices as
  \begin{equation} \label{eq:wedge}
    \eval{\alpha \wedge \beta}{\simplex{k+l}} =
    \dfrac{1}{(k+l+1)!}\sum_{\tau \in S_{k+l+1}} \sgn(\tau) \eval{\alpha \smile \beta}
    {\simplex[\tau(0)]{\tau(k+l)}}.\\
  \end{equation}
\end{definition}

\begin{example}[Wedge product on a triangle]\label{eg:wedge1-1}
Let $X$ be the oriented triangle $[012]$ and $\alpha,\beta \in C^1(X)$. Using Definition~\ref{defn:wedge}
\begin{multline}\label{eq:wedge012}
    \eval{\alpha\wedge \beta}{[012]} = \dfrac{1}6 
    \bigg[\eval{\alpha\smile \beta}{[012]} - \eval{\alpha \smile \beta}{[021]}
    -\eval{\alpha\smile \beta}{[102]} \\+ \eval{\alpha\smile \beta}{[120]}
    + \eval{\alpha\smile \beta}{[201]} - \eval{\alpha\smile \beta}{[210]}
    \bigg]\, . 
\end{multline}
Using the shorthand notation $\alpha_{ij}$ for $\eval{\alpha}{[ij]}$ and $\beta_{jk}$ for $\eval{\beta}{[jk]}$, terms like $\eval{\alpha\smile \beta}{[ijk]}$ above can be written as $\alpha_{ij} \beta_{jk}$ with an appropriate sign depending on the sign of the permutation corresponding to the ordering $i,j,k$ of vertices. Then~\eqref{eq:wedge012} is
\[
\eval{\alpha\wedge \beta}{[012]} = \dfrac{1}6 \big[\alpha_{01} \beta_{12} - \alpha_{02} \beta_{21} - \alpha_{10} \beta_{02} + \alpha_{12} \beta_{20} +
\alpha_{20} \beta_{01} -  \alpha_{21} \beta_{10}\big]\, .
\]
The terms can be collected by vertices and the signs adjusted to yield
\[
\eval{\alpha\wedge \beta}{[012]} = \dfrac{1}6 \big[(\alpha_{01} \beta_{02} - \alpha_{02} \beta_{01}) + (\alpha_{01} \beta_{12} -  \alpha_{12} \beta_{01}) +
(\alpha_{02} \beta_{12} -  \alpha_{12} \beta_{02})\big]\, ,
\]
where the terms on the RHS can be interpreted as alternating products at the three vertices. Such an interpretation has existed at least since~\cite{Hirani2003}. One of the results in this paper is an alternative, averaging interpretation of the discrete wedge product; see Proposition~\ref{prop:averaging}.

\begin{remark}
The discrete $d$ satisfies a Leibniz rule with respect to the discrete $\wedge$ since the coboundary operator does so with respect to $\smile$. It is known that the discrete $\wedge$ is anti-commutative but not associative~\cite{Hirani2003}. This lack of associativity is encoded by an $A_\infty$-algebra structure  on the cochains of a simplicial complex~\cite{DoMoSh2008}. Together with the skew-commutativity of the wedge product, one in fact obtains a $C_\infty$-algebra structure~\cite{Sullivan1977,TrZeSu2007,Wilson2007}. Example~\ref{ex:associative} provides a simple computation demonstrating the failure of associativity.
\end{remark}
    
Wilson~\cite{Wilson2007} defined a cochain product and proved convergence and other properties for this product. Wilson also used this product to define a combinatorial Hodge star operator. The next proposition shows that combinatorial DEC wedge product of Definition~\ref{defn:wedge} is equal to Wilson's cochain product. 
Wilson's cochain product uses the space of Whitney forms~\cite{Dodziuk1974}  on the underlying space of the simplicial complex $X$ (denoted $\mathcal{P}_1^-\Lambda^k(X)$ in the notation of~\cite{ArFaWi2010}) and the family of Whitney interpolation maps $W: C^k(X) \to \mathcal{P}_1^-\Lambda^k(X)$ which map cochains to piecewise polynomial differential forms. 

\begin{proposition} \label{prop:wilson}
Let $X$ be a triangulated manifold, $\alpha\in C^k(X)$, $\beta \in C^l(X)$ and $\nu$ a $(k+l)$-dimensional simplex in $X$. Then 
\[
\eval{\alpha \wedge \beta}{\nu} = \int_\nu W\alpha \wedge W\beta\, .
\]
\end{proposition}
Here the $\wedge$ on the left is the combinatorial DEC wedge product of Definition~\ref{defn:wedge} and the one on the right is the wedge product on smooth forms. The RHS above is Wilson's cochain product~\cite[Definition 5.1]{Wilson2007}.
\begin{proof}
    Let $\nu = \simplex{k+l}$ and assume all lower dimensional faces of $\nu$ are oriented in some arbitrary way, for example according to increasing vertex ordering. Identifying cochains and chains let $\alpha = \sum_{\sigma^k\prec \nu} \alpha_\sigma \sigma$ where $\sigma^k \prec \nu$ means $\sigma$ is a $k$-dimensional face of $\nu$ and $\alpha_\sigma \in \R$. Similarly $\beta = \sum_{\tau^l\prec \nu} \beta_\tau \tau$. Then 
    \begin{equation*}
        \int_\nu W\alpha \wedge W\beta = \sum_{\substack{\sigma^k, \tau^l \prec \nu\\ \sigma\cdot \tau = \nu}} \alpha_\sigma \beta_\tau \int_\nu W\sigma \wedge W\tau\, .    
    \end{equation*}
    Here the sum is over all $k$-dimensional faces $\sigma$ and $l$-dimensional faces $\tau$ of $\nu$ such that $\sigma$ and $\tau$ intersect in exactly one vertex and span $\nu$. (We denote this spanning property by $\sigma \cdot \tau = \nu$.) This is because the integrand on the right is zero if $\sigma$ and $\tau$ intersect in more than one vertex. See the proof of~\cite[Theorem 5.2]{Wilson2007}. Also by~\cite[Theorem 5.2]{Wilson2007}, if $\sigma = \simplex[\rho(0)]{\rho(k)}$, $\tau = \simplex[\rho(k)]{\rho(k+l)}$, $\rho \in S_{k+l+1}$ and $\nu' = \simplex[\rho(0)]{\rho(k+l)}$ then
    \[
    \int_{\nu'} W\sigma \wedge W\tau = \varepsilon(\sigma,\tau)
    \frac{k!\; l!}{(k+l+1)!}\, ,
    \]
    where $\varepsilon(\sigma,\tau)$ is a sign determined by $\textrm{orientation}(\sigma)\cdot\textrm{orientation}(\tau) = \varepsilon(\sigma,\tau)\cdot \textrm{orientation}(\nu')$. Without loss of generality, we will use increasing vertex ordering to orient the faces and let $\sgn(\sigma)$ denote the sign of the permutation needed to bring the vertices of $\sigma$ into increasing order etc. Then the equation for $\varepsilon$ is $\sgn(\sigma) \cdot \sgn(\tau) = \varepsilon(\sigma,\tau) \cdot \sgn(\nu')$. Then
    \[
    \int_\nu W\alpha \wedge W\beta = \sum_{\substack{\sigma^k, \tau^l \prec \nu\\ \sigma\cdot \tau = \nu}} \alpha_\sigma \beta_\tau\varepsilon(\sigma,\tau)
    \frac{k!\; l!}{(k+l+1)!}\, .
    \]
    On the other hand, the terms in $\eval{\alpha \wedge \beta}{\nu}$ are of the form $\sgn({\rho})\,\alpha_{\sigma}\beta_{\tau}$ where $\sigma = \simplex[\rho(0)]{\rho(k)}$, $\tau = \simplex[\rho(k)]{\rho(k+l)}$ and $\rho \in S_{k+l+1}$. There are $k! \, l!$ such terms since the last vertex of $\sigma$ equals the first vertex of $\tau$. All of these terms acquire the same sign $\varepsilon(\sigma,\tau)$ when the vertices of $\sigma$ and $\tau$ are permuted to bring them into increasing vertex order. The normalizing factor in Definition~\ref{defn:wedge} is the denominator above.
 \end{proof}

\begin{remark}
    As a corollary of Proposition~\ref{prop:wilson} and~\cite[Theorem 5.4]{Wilson2007} the combinatorial wedge product of DEC converges to the wedge product of smooth forms in the sense of \cite[Theorem 5.4]{Wilson2007}. This proposition also opens 
    up a path to connect DEC to the extensive literature on algebraic structures on simplicial cochains including to $C_\infty$-algebras~\cite{Wilson2007} and $A_\infty$-algebras~\cite{DoMoSh2008}.
\end{remark}

\section{Naturality of Exterior Derivative and Wedge Product}
\label{sec:naturality}
We first recall naturality of the exterior calculus operations $d$ and $\wedge$ on smooth manifolds. Let $M$ and $N$ be smooth manifolds, $F: M \to N$ a smooth map and $\alpha, \beta \in \Lambda^\bullet(N)$ differential forms on $N$. Then naturality of $d$ means that $F^* d = d F^*$, and naturality of $\wedge$ means $F^* (\alpha \wedge \beta) = F^*\alpha \wedge F^*\beta$. Recall that $F^*: \Lambda^\bullet(N) \to \Lambda^\bullet(M)$ is the pullback operator defined by $F^* \alpha (v_1,\ldots, v_k) = \alpha(F_{*,p}v_1, \ldots, F_{*,p}v_k)$ for a $k$-form $\alpha$ and $v_1,\ldots,v_k \in T_pM$ for all $p\in M$. At every point $p\in M$, the linear map $F_{*,p}: T_p M \to T_{F(p)} N$ between tangent spaces is the pushforward or differential which in coordinates in the Jacobian matrix computed at $p$. See \cite{Tu2011} for a review of these concepts.

An example of naturality of $d$ is the chain rule in single variable calculus. For $f, g: \R \to \R$, $f(g(x))' = f'(g(x)) g'(x)$. The LHS is $d(f \circ g)_x = d(g^* f)_x$ and the RHS is $(g^* df)_x$ since $f'(g(x)) g'(x) = df_{g(x)} g_{*, x}$. Another application of naturality of $d$ is in proving the diffeomorphism invariance of cohomology groups. The naturality of $\wedge$ is used, for example, in change of coordinate computations.

We will define a pullback of cochains induced from abstract simplicial maps and show that the discrete $d$ and $\wedge$ commute with this pullback. First we recall the definition of abstract simplicial maps and a homomorphism on chains induced from such maps. See~\cite{Munkres1984,Lee2011} for more details.

\begin{definition}
  For simplicial complexes $X$ and $Y$ an \emph{abstract simplicial map} $f\colon  X\to Y$ is given by the data of a map of sets, $f^{(0)}: X^{(0)} \to Y^{(0)}$, with the property that if $\{u_0, \ldots u_k\}$ spans a simplex in $X$, the set $\{f(u_0), \ldots, f(u_k)\}$ spans a simplex in $Y$. The map $f^{(0)}$ is  called the \emph{vertex map of $f$}. 
\end{definition}
For us, abstract simplicial maps will be maps between oriented simplices. That is, the sets $\{u_0, \ldots, u_k\}$ (simplices) above are replaced by ordered sets $[u_0, \ldots, u_k]$ (oriented simplices). 

\begin{remark}
  An abstract simplicial map can collapse simplices. But the spanning property of vertex maps implies that vertices connected by an edge do not lose that property of being ``discretely near''. Intuitively, vertices can move closer, but cannot move far apart. Thus, abstract simplicial maps are a combinatorial proxy for smooth (or at least continuous) maps.
\end{remark}

\begin{example}
    Let $X$ be the simplicial complex that is the boundary of a triangle with vertices $u_0, u_1, u_2$ and $Y$ the simplicial complex with vertices $v_0, v_1, v_2$ formed by the two edges $[v_0, v_1]$ and $[v_1, v_2]$. The vertex map $u_i \mapsto v_i$, for $i=0,1,2$ does \emph{not} define an abstract simplicial map because $[u_0, u_2]$ is an edge in $X$ but $[v_0, v_2]$ is not an edge in $Y$. The vertices $u_0$ and $u_2$ that were ``nearby" in the sense of being connected by an edge have become further apart after the mapping.
\end{example}
\begin{example}\label{ex:collapse}
    Now consider the same $X$ and let $Y$ be a simplicial complex consisting of the single edge $[v_0, v_1]$. The vertex map $u_i \mapsto v_i$, for $i=0,1$ and $u_2 \mapsto v_0$ \emph{does} define an abstract simplicial map and in this map the edge $[u_0, u_2]$ collapses to the vertex $v_0$. 
\end{example}
   
Next we review a standard construction in algebraic topology that builds a homomorphism on chains from an abstract simplicial map. See~\cite{Munkres1984} for the applications and properties of this homomorphism. 

\begin{definition}
  Let $f : X \to Y$ be an abstract simplicial map between simplicial complexes $X$ and $Y$. Define a homomorphism $f_\sharp : C_k(X) \to C_k(Y)$ for each $k$ determined by the values on oriented simplices,
  \[
    f_\sharp([u_0,\ldots,u_k]) =
    \begin{cases}
      [f(u_0),\ldots, f(u_k)], & \text{if } f(u_0), \ldots, f(u_k) \text{ distinct}\\
      0 & \text{otherwise}
    \end{cases}
  \]
\end{definition}
It is easy to see that $f_\sharp$ is well-defined since both sides change sign according to permutation of the vertices. We use this to define the pullback of cochains.
\begin{definition}
  Given $f : X \to Y$ an abstract simplicial map and a cochain $\alpha \in C^k(Y)$ the \emph{pullback} $f^*\alpha \in C^k(X)$ is defined by its values $f^*\alpha(c) := \alpha(f_\sharp(c))$ for chains $c \in C_k(X)$.
\end{definition}

\begin{proposition}[Naturality of discrete $d$] \label{prop:dnaturality}
   Let $f: X \to Y$ be an abstract simplicial map between simplicial complexes $X$ and $Y$. Then discrete $d$ commutes with the pullback along $f$: for any $\alpha\in C^k(Y)$ we have $f^*d\alpha=d (f^*\alpha)$. Equivalently, 
  \begin{equation} \label{eq:dnaturality}
    \eval{f^\ast d\alpha}{\simplex[u_0]{u_{k+1}}} =
    \eval{d f^\ast\alpha}{\simplex[u_0]{u_{k+1}}}
  \end{equation}
   for all $(k+1)$-simplices $\simplex[u_0]{u_{k+1}}$ in $X$.
\end{proposition}
\begin{proof}
This follows immediately from the definitions of $d$ and $f^*$ and the fact that $\partial f_\sharp = f_\sharp \partial$. See~\cite[Lemma 12.1]{Munkres1984} for a proof of this fact.
\end{proof}

\begin{example}
Let $X$ and $Y$ be the simplicial complexes of Example~\ref{ex:collapse}, $f$ the simplicial map defined in that example and let $\alpha \in C^0(X)$ be the 0-cochain that takes the value $\alpha_0$ and $\alpha_1$ on the vertices $v_0$ and $v_1$. Then the pullback $f^*\alpha$ takes the values $\alpha_0, \alpha_1$ and $\alpha_0$ on the vertices $u_0, u_1$ and $u_2$, respectively. Thus the values of the $df^*\alpha$ on the three edge of $X$ are $\eval{df^*\alpha}{[u_0,u_1]} = \alpha_1 - \alpha_0$,  $\eval{df^*\alpha}{[u_0,u_2]} = \alpha_0 - \alpha_0 = 0$, and $\eval{df^*\alpha}{[u_1,u_2]} = \alpha_0 - \alpha_1$. On the other hand the evaluations of the pullback of $d\alpha$ are $\eval{f^*d\alpha}{[u_0, u_1]} = \eval{d\alpha}{[v_0, v_1]} = \alpha_1 - \alpha_0$, $\eval{f^*d\alpha}{[u_0, u_2]} = \eval{d\alpha}{[v_0]} = 0$, and $\eval{f^*d\alpha}{[u_1, u_2]} = \eval{d\alpha}{[v_1, v_0]} = \alpha_0 - \alpha_1$, verifying the naturality of $d$ in this example.
\end{example}


\begin{proposition}[Naturality of discrete wedge product]\label{prop:wedge_naturality}
Let $f: X \to Y$ be an abstract simplicial map between simplicial complexes $X$ and $Y$. Then the discrete $\wedge$ commutes with the pullback along $f$: for all $\alpha \in C^k(Y)$ and $\beta \in C^l(Y)$, we have $f^*(\alpha\wedge\beta)=f^*\alpha\wedge f^*\beta$. Equivalently, 
  \begin{equation}\label{eq:wedge_naturality}
   \eval{f^\ast(\alpha \wedge \beta)}{\simplex[u_0]{u_{k+l}}}=
    \eval{f^\ast\alpha \wedge f^\ast \beta}{\simplex[u_0]{u_{k+l}}}
    \end{equation}  
    for all $(k+l)$-simplices $\simplex[u_0]{u_{k+l}}$ in $X$.
\end{proposition}
\begin{proof}
    We note that the statement~\eqref{eq:wedge_naturality} is for a simplex but extends by linearity to chains. The analogous statement for the cup product follows directly from the definitions, 
  \begin{align*}
    \eval{f^\ast(\alpha \smile \beta)}{\simplex[u_0]{u_{k+l}}} &=
                                                         \eval{\alpha \smile \beta}{f(\simplex[u_0]{u_{k+l}})}=
    \eval{\alpha \smile \beta}{\simplex[f(u_0)]{f(u_{k+l})}}\\
                                                                 &= \eval{\alpha}{\simplex[f(u_0)]{f(u_k)}}\;\;
                                                                   \eval{\beta}{\simplex[f(u_k)]{f(u_{k+l})}}\\
    &= \eval{f^\ast\alpha}{\simplex[u_0]{u_k}}\eval{f^\ast \beta}{\simplex[u_k]{u_{k+l}}}\\
    &= \eval{f^\ast\alpha \smile f^\ast \beta}{\simplex[u_0]{u_{k+l}}}\, .
  \end{align*}
  We observe that (for dimension reasons) both sides are 0 if the vertex map of $f$ is not a bijection when restricted to $\{u_0,u_1,\dots, u_{k+l}\}$. 
  
  Adapting the above computation to the wedge product, the terms in the expansion of $\eval{\alpha \wedge \beta}{\simplex[f(u_0)]{f(u_{k+l})}}$ are of the form $\eval{\alpha \smile \beta}{\simplex[f(u_{\tau(0)})]{f(u_{\tau(k+l)})}}$ where $\tau\in S_{k+l+1}$ is a permutation. If the vertex map of $f$ is a bijection, then each such term is equal to
  \begin{equation}\label{eq:tau_term}
    \eval{\alpha}{\simplex[f(u_{\tau(0)})]{f(u_{\tau(k)})}}\;\;
    \eval{\beta}{\simplex[f(u_{\tau(k)})]{f(u_{\tau(k+l)})}}
  \end{equation}
  by the cup product result by relabelling the vertices under the permutation $\tau$.

  If the vertex map of $f$ is not a bijection on $\{u_0,u_1,\dots, u_{k+l}\}$, then the LHS of~\eqref{eq:wedge_naturality} is 0. To show that the RHS is also 0, assume that for some $i \ne j$, $f(u_i) = f(u_j)$. If both $i$ and $j$ are in $\{\tau(0),\ldots,\tau(k)\}$ or both are in $\{\tau(k),\ldots,\tau(k+l)\}$ then that particular term of the form~\eqref{eq:tau_term} is 0 for dimensional reasons.

  Next assume that $i =\tau(a)$ and $j=\tau(b)$ for $0\le a \le k$ and $k \le b \le k+l$ so that the term of type~\eqref{eq:tau_term} is not automatically 0. In this case, there will be a matching term in which $j=\tau(a)$ and $i=\tau(b)$. These two terms are identical and appear with opposite signs $\sgn(\tau)$ and hence cancel.
\end{proof}

In addition to naturality, the discrete $d$ and $\wedge$ satisfy a Leibniz rule~\cite{Hirani2003}. This is proved here for completeness.
\begin{proposition}[Leibniz rule]\label{prop:Leibniz}
  Let $\alpha \in C^{k}(X)$ and $w \in C^l(X)$ and $c \in C_{k+l}(X)$. Then
    \[
    \eval{d (\alpha \wedge w)}{c}=
    \eval{d \alpha \wedge w +   (-1)^k \alpha\wedge d w}{c}\, .
    \]
\end{proposition}
\begin{proof}
It is enough to show this for a $(k{+}l{+}1)$-simplex and then extend by linearity to chains. Each element of the permutation group $S_{k{+}l{+}2}$ acts as an isomorphism on a $(k{+}l{+}1)$-simplex $\sigma$. Each permutation of vertices defines a vertex map which corresponds to an abstract simplicial isomorphism. Using $\tau$ to refer to an element of $S_{k{+}l{+}2}$  as well as the corresponding simplicial isomorphism, note that $\tau$ commutes with the boundary operator on chains. Thus
\begin{align*}
  \eval{d(\alpha \wedge w)}{\sigma} = \eval{\alpha \wedge w}{\partial \sigma}
  & = \sum_\tau\sgn(\tau) \eval{\alpha \smile w}{\tau \partial \sigma}\\
  & = \sum_\tau \sgn(\tau) \eval{\alpha \smile w}{\partial \tau \sigma} =
    \sum_\tau\sgn(\tau) \eval{d(\alpha \smile w)}{\tau \sigma}\, ,
\end{align*}        
which, by Leibniz rule for cup products is
\[  
  \sum_\tau \sgn(\tau) \eval{d\alpha \smile w+
      (-1)^{\vert \alpha \vert} \alpha \smile d w}{\tau\sigma} =
        \eval{d\alpha \wedge w+ (-1)^{\vert \alpha \vert} \alpha     \wedge d w}{\sigma}\, .
\]

\end{proof}
It is illustrative to see this in the following simplest example.
\begin{example} Let $f,g \in C^0(X)$, where $X$ is the edge $[01]$. We will denote evaluation of $f$ on vertex $i$ as $f_i$ etc. Now we check that 
\[   \eval{d(f\wedge g)}{[01]} =
    \eval{df\wedge g + f\wedge dg}{[01]}\, .
\]
The LHS is $f_1 g_1 - f_0 g_0$ and so is the RHS since
  \begin{align*}
    \eval{df \wedge g}{[01]} &= 
    \frac{1}2\big[ \eval{df}{[01]} \, g_1 - \eval{df}{[10]} \, g_0 \big] = \frac{1}2\big[(f_1 - f_0) g_1 - (f_0 - f_1) g_0\bigr]\\
    \eval{f\wedge dg}{[01]} &= 
    \frac{1}2\big[f_0\, \eval{dg}{[01]} - f_1 \,\eval{dg}{[10]}\big]
    = \frac{1}2 \big[f_0(g_1 - g_0) - f_1 (g_0 - g_1)\big]\, .
  \end{align*}
\end{example}

\section{Averaging Interpretation of Wedge Product}

 The main result of this section shows that the combinatorics of the discrete wedge product in DEC organizes into a neat averaging formula; see Proposition~\ref{prop:averaging}. We begin with motivation from several low-dimensional examples that illustrate this averaging interpretation. 

\begin{example}
Let $X$ be the oriented simplicial complex with a single edge $[01]$, and consider $f \in C^0(X)$ and $\alpha \in C^1(X)$. Let $f_i$ denote the values of $f$ on vertex $i$ and $\alpha_{01}$ the evaluation of $\alpha$ on the edge $[01]$. Then the evaluation of $f\alpha = f\wedge \alpha$ on the edge $[01]$ is a weighted average of the values of $f$, 
\begin{align*}
\eval{f \wedge \alpha}{[01]} = \frac{1}2 \big[\eval{f\smile \alpha}{[01]} - \eval{f\smile \alpha}{[10]}\big] = 
\frac{1}2\big[f_0\; \alpha_{01} - f_1 \; \alpha_{10}\big] = 
\frac{f_0 + f_1}2 \; \alpha_{01}\, .
\end{align*}
\end{example}


\begin{example}
Example~\ref{eg:wedge1-1} reviewed the standard alternating product interpretation of the discrete wedge product of two 1-cochains. This was achieved by collecting terms at each vertex. However, one can also collect the terms by edges, yielding a sum of weighted averages
\begin{equation} \label{eq:averaged012}
\eval{\alpha\wedge \beta}{[012]} =\dfrac{1}3 \bigg[ \alpha_{01}\dfrac{\left(\beta_{02} + \beta_{12}\right)}2 + 
\alpha_{12}\dfrac{\left(\beta_{10} + \beta_{20}\right)}2 +
\alpha_{20}\dfrac{\left(\beta_{01} + \beta_{21}\right)}2
\bigg]\, .
\end{equation}
In words,~\eqref{eq:averaged012} shows that the value of the wedge product of 1-cochains $\alpha$ and $\beta$ on a triangle comes from going around the triangle multiplying the value of $\alpha$ on an edge by the average value of $\beta$ on the two edges incident on that first edge. This is done for all 3 edges and the final result is the average of these products. Alternatively one can reverse the roles of $\alpha$ and $\beta$ (in terms of inner averaging and outer averaging) and then 
\begin{equation} \label{eq:averaged012other}
  \eval{\alpha\wedge \beta}{[012]} =\dfrac{1}3
  \bigg[
  \dfrac{\left(\alpha_{20} +  \alpha_{21}\right)}2 \, \beta_{01} +
  \dfrac{\left(\alpha_{01} + \alpha_{02}\right)}2\, \beta_{12} +
  \dfrac{\left(\alpha_{10} + \alpha_{12}\right)}2\, \beta_{20}
  \bigg]\, .
\end{equation}
We emphasize that this averaging interpretation depends on the choice of orientation of chains. For example, $\beta_{02}$ is used in the first term in~\eqref{eq:averaged012} while $\beta_{20}$ is used in the second term and $\alpha_{20}$ rather than $\alpha_{02}$ is used in the third term. Informally, in this case one goes around the triangle in the direction it is oriented (counterclockwise in this case) taking values of $\alpha$ on each edge and the values of $\beta$ being used are taken on the two remaining edges pointing away. For~\eqref{eq:averaged012other} one takes the values with edges pointing towards. Part of the content of Proposition~\ref{prop:averaging} is to show that such choices permitting an averaging interpretation always exist.
\end{example}
\end{example}

\begin{example}
  The case of $\alpha \in C^0(X)$ and $\beta \in C^2(X)$ also demonstrates the simplicity of the averaging interpretation.
  \begin{align*}
    \eval{\alpha\wedge \beta}{[012]} &= \left(\dfrac{\alpha_0 + \alpha_1 + \alpha_2}3\right) \beta_{012}
\end{align*}
\end{example}

To build intuition for the general case proved in Proposition~\ref{prop:averaging}, consider the following example.
\begin{example}\label{eg:wedge2-1}
Let $\alpha \in C^2(X)$ and $\beta \in C^1(X)$ for $X$ the tetrahedron $[0123]$.  The averaging interpretation will yield an average over 4 terms, one for each triangle of the tetrahedron. Each term will be a product of the value of $\alpha$ on a triangle multiplied by the average of the 3 values of $\beta$ corresponding to the other 3 edges of the tetrahedron touching the triangles at the vertices of the triangle. Explicitly,
\begin{multline}\label{eq:wedge2-1}
  \eval{\alpha \wedge \beta}{[0123]} = \dfrac{1}4
  \Big[
  \alpha_{012}\left(\dfrac{\beta_{03}+ \beta_{13} + \beta_{23}}{3}\right) +
  \alpha_{031}\left(\dfrac{\beta_{02}+ \beta_{12} + \beta_{32}}{3}\right) \\+
  \alpha_{023}\left(\dfrac{\beta_{01}+ \beta_{21} + \beta_{31}}{3}\right) +
  \alpha_{132}\left(\dfrac{\beta_{10}+ \beta_{20} + \beta_{30}}{3}\right)
  \Big]. \end{multline}
The roles of $\alpha$ and $\beta$ in inner and outer averaging could have been reversed as in~\eqref{eq:averaged012other} as compared with~\eqref{eq:averaged012}.  Notice again that particular choices of orientations have been used in order to achieve the averaging interpretation in~\eqref{eq:wedge2-1}. All triangle terms are on triangles taken counter-clockwise viewed from outside the tetrahedron and all edge terms are on edges going away from the triangle. To see what becomes of the factor 1/24 = 1/(2+1+1)! in Definition~\ref{defn:wedge}, after collecting the terms by vertices the terms can be arranged as
\begin{multline*}
    \frac{1}{24}\Big[
    2(\alpha_{012}\beta_{03} - \alpha_{013}\beta_{02} + \alpha_{023}\beta_{01}) +
    2(\alpha_{012}\beta_{13} - \alpha_{013}\beta_{12} + \alpha_{123}\beta_{01}) +\\
    2(\alpha_{012}\beta_{23} - \alpha_{023}\beta_{12} + \alpha_{123}\beta_{02}) +
    2(\alpha_{013}\beta_{23} - \alpha_{023}\beta_{13} + \alpha_{123}\beta_{03}) 
    \Big]\, .
\end{multline*}
Thus the factor outside becomes 1/12 which finally appears in equation~\eqref{eq:wedge2-1} as (1/4)(1/3) with (1/4) for the outer averaging over the 4 triangles of the tetrahedron and (1/3) for the inner averaging over the 3 other edges touching each triangle.
\end{example}

With the above motivating examples in place, we turn to the general proof of the averaging interpretation of the discrete wedge product. Let $k,l \in \mathbb{Z}_{\ge 0}$ and $\sigma = \{0, 1, \ldots, k+l\}$ and let $S_{k+l+1}$ be the group of permutations of the elements of $\sigma$. Let $i: S_k \hookrightarrow S_{k+l+1}$ and $j:S_l \hookrightarrow S_{k+l+1}$ be inclusions so that elements of $i(S_k)$ and $j(S_l)$ act on the first $k$ and last $l$ positions of the input, respectively. For any $F \subset \sigma$ with $k$ elements, $v \in \sigma\setminus F$ and $G = \sigma \setminus (F \cup \{v\})$ define the set of permutations
\[
  P(F,v,G) \defeq \setdef{\rho}{\rho \in S_{k+l+1}, \; \rho(0), \ldots, \rho(k-1) \in F,
    \; \rho(k) = v, \; \rho(k+1),\ldots,\rho(k+l)\in G}\, .
\]
We will denote by $(f,v,g)$ an ordering of the elements of $\sigma$ such that as sets, $f = F$ and $g = G$.  By definition of $P(F,v,G)$, each such ordering corresponds uniquely to a permutation in $P(F,v,G)$. We will write the action of a permutation $\rho \in S_{k+l+1}$ on an ordering $(f,v,g)$ as $(\rho(f), v, \rho(g))$ and the action of $\tau_k \circ \tau_l$ for $\tau_k \in i(S_k)$ and $\tau_l \in j(S_l)$ as  $(\tau_k(f),v,\tau_l(g))$.
  
\begin{lemma}\label{lem:permutation}
  Let $F,v,G$ be as above and $(f_0, v, g_0)$ an ordering corresponding to a particular chosen permutation $\tau \in P(F,v,G)$. Then for all $(f,v,g)$ orderings corresponding to permutations in $P(F,v,G)$ there exist $\tau_k \in i(S_k)$ and $\tau_l \in j(S_l)$ depending on $(f,v,g)$ such that
  \begin{enumerate}[label=(\roman*)]
  \item   $(\tau_k(f), v, \tau_l(g)) = (f_0, v, g_0)$; and
  \item $\sgn(\tau_k) \sgn(\tau_l) = \sgn(\tau)$.
  \end{enumerate}
\end{lemma}
\begin{proof}
Let $\eta \in S_{k+l+1}$ such that $(\eta(f_0), v, \eta(g_0)) = (0,\ldots, k+l)$. Then $\tau = \eta^{-1} \, \tau_l \, \tau_k\,  \eta$. Also, since the sign homomorphisms from $S_k, S_l, S_{k+l+1}$ to $Z_2$ are compatible with the inclusions $i$ and $j$ we have that $\sgn(\tau) = \sgn(\eta^{-1} \, \tau_k \, \tau_l \, \eta) = \sgn(\tau_k) \, \sgn(\tau_l)$.
\end{proof}
\begin{notation}
In the following we use the notation $f ^k \prec \sigma$ to denote a $k$-face $f$ of simplex $\sigma$, \emph{ignoring} orientations. That is, $f^k$ is just a subset of size $k+1$ of the vertices of $\sigma$. Sometimes we skip the superscript to simplify notation. For $\sigma$ an oriented simplex and $f$ an oriented face or a set of vertices, the face $\sigma \setminus f$ is an oriented face of $\sigma$ formed by deleting the vertices in face $f$ from the vertices of $\sigma$. For $g$ an oriented simplex and vertex $v$, $v \ast g$ is an  oriented simplex formed by union of $\{v\}$ with the vertex set of $g$.  As before, whenever a simplex is used in an evaluation of a cochain, for example $f$ in $\eval{\alpha}{f}$ it is assumed to be oriented. The specific orientation being used is not apparent in this notation.
\end{notation}

\begin{proposition}[Averaging interpretation of discrete wedge product]\label{prop:averaging}
  Let $\alpha \in C^k(X)$, $\beta \in C^l(X)$ and $\sigma$ be a $(k+l)$-simplex in $X$. Then the discrete wedge product of Definition~\ref{defn:wedge} is
  \begin{align}
    \eval{\alpha\wedge \beta}{\sigma} &=
                                      \frac{1}{\binom{k+l+1}{k+1}}
                                      \sum_{f^k \prec \sigma} \eval{\alpha}{f}\quad  
                                      \bigg(\frac{1}{k+1} \sum_{v^0 \prec f}
                                      \eval{\beta}{v \ast (\sigma \setminus f)} \bigg) \label{eq:outer_alpha}\\
                                    &=
                                      \frac{1}{\binom{k+l+1}{l+1}}
                                      \sum_{f^l \prec \sigma}\, 
                                      \bigg(\frac{1}{l+1} \sum_{v^0 \prec f}
                                      \eval{\alpha}{v \ast(\sigma\setminus f)}\bigg)\quad \eval{\beta}{f}\, ,
                                      \label{eq:outer_w}
  \end{align}
  where the orientations of $f$ and $v \ast (\sigma \setminus f)$ are such that the ordering $(f\setminus \{v\}, v, \sigma\setminus f)$ corresponds to an even permutation in $S_{k+l+1}$.
\end{proposition}
\begin{proof}
  We prove the first equality. The proof for the second is similar. The RHS of~\eqref{eq:outer_alpha} can be written as a double sum, first summing over all vertices and for a fixed vertex summing over all the $(k-1)$-faces of $\sigma$ not containing that vertex to get
  \[
    \eval{\alpha\wedge \beta}{\sigma} =\frac{1}{k+l+1!} \; \sum_{v^0 \prec \sigma}
    \quad \sum_{g^{k-1} \prec (\sigma\setminus \{v\})} k! \, \eval{\alpha }{v\ast g}
    \quad l! \, \eval{\beta}{\sigma\setminus g}\, .
  \]
  The orientations of the simplices $v\ast g$ and $\sigma\setminus g$ in the cochain evaluations above are such that the ordering $(g, v, \sigma \setminus (g\cup \{ v\}))$ corresponds to an even permutation in $P (g, v, \sigma \setminus (g\cup \{ v\}))$. (We have used $g$ etc. to represent both an ordering of vertices and the corresponding set.) The fact that the orderings of $\sigma$ corresponding to all the permutations in $P (g, v, \sigma \setminus (g\cup \{ v\}))$ can be reordered to the ordering $(g, v, \sigma \setminus (g\cup \{ v\}))$  follows from Lemma~\ref{lem:permutation}. The $k!\, l!$ factorial follows from the fact that there are $k!$ orderings for $g^{k-1} \prec (\sigma \setminus \{v\})$ once a $g$ is fixed and after in addition fixing a $v$ there are $l!$ orderings for the remaining vertices. The above can be rewritten as 
  \[
    \frac{1}{k+1} \frac{1}{\binom{k+l+1}{k+1}} \; \sum_{v^0\prec \sigma}
    \quad \sum_{v^0 \ast g^{k-1} \prec \sigma} \eval{\alpha}{v\ast g}
    \quad \eval{\beta}{v \ast (\sigma \setminus (v \ast g))}\, .
  \]
Renaming $v\ast g \eqdef f$ we can rewrite this as 
  \[
    \frac{1}{k+1} \frac{1}{\binom{k+l+1}{k+1}} \;
    \sum_{v^0\prec \sigma} \quad \sum_{f^k \prec \sigma} \eval{\alpha}{f}
    \quad \eval{\beta}{v \ast (\sigma \setminus f)}\, ,
  \]
  where it is understood that $f$ is a $k$-face of $\sigma$ that must contain the vertex $v$.  This can be expressed equivalently by using Lemma~\ref{lem:permutation} once for every choice of $v$ and switching the summation and moving the normalizing factors as
  \[
    \frac{1}{\binom{k+l+1}{k+1}} \;
    \sum_{f^k \prec \sigma} \eval{\alpha}{f}\quad 
    \bigg(\frac{1}{k+1}\sum_{v^0 \prec f} \eval{\beta}{v \ast (\sigma \setminus f)}\bigg)\, ,    
  \]
  where now the vertex summation is over all vertices $v$ in $f$.

It is crucial to note here that the only reason that we have been able to collect all the $\beta$ evaluations with a single $\alpha$ evaluation is because the Lemma~\ref{lem:permutation} can be used once for every choice of $v$ once the face $f$ and its orientation have been fixed.
\end{proof}

The following is Example 5.8 in~\cite{Wilson2007} and demonstrates the well known fact that the DEC wedge product and Wilson's cochain product are not associative. We use the averaging interpretation for the computation.
\begin{example}\label{ex:associative}
    Let $X$ be the edge $[01]$, $\alpha, \beta \in C^0(X)$. Let $\alpha_0, \beta_0 = 1,0$, $\alpha_1, \beta_1 = 0,1$ and $\omega_{01} = 1$. Then 
    \[
    \Eval{(\alpha \wedge \beta)\wedge \omega}{[01]} = 
    \frac{\alpha_0 \beta_0 + \alpha_1 \beta_1}2 \, \omega_{01} = 0\, .
    \]
    On the other hand,
    \[
    \Eval{\alpha \wedge (\beta\wedge \omega)}{[01]} = 
    \frac{\alpha_0 + \alpha_1}2
    \left(\frac{\beta_0+\beta_1}2 \, \omega_{01}\right) = \frac{1}4\, .
    \]
\end{example}

\section{Conclusion}
In this paper we have developed a categorical perspective on Discrete Exterior Calculus. The earlier work on DEC focused on discretizing the \emph{objects} of exterior calculus (namely differential forms), operators on forms and domains on which these form live. In contrast, in Section~\ref{sec:naturality} we have promoted a discretization of \emph{morphisms} of exterior calculus (namely smooth maps). We also gave an intuitive averaging interpretation of the discrete wedge which generalizes an averaging interpretation that was known in a special case of the discrete wedge of a 1-cochain with 0-cochain.

As proved in Proposition~\ref{prop:wilson}, the anti-symmetrized cup product that is used as a discrete wedge product in DEC is equal to the cochain product of Wilson~\cite{Wilson2007}. Then by a theorem of Sullivan~\cite{TrZeSu2007} (see also \cite[Theorem~5.11]{Wilson2007}) there exists a local construction which canonically extends the algebra of DEC to a $C_\infty$-algebra. 
Under mesh refinement, this $C_\infty$-algebra converges to the algebra given by wedge product on forms, see~\cite[Theorem 5.12]{Wilson2007} for details. This suggests a direction for making the combinatorial wedge product of DEC more accurate by including more and more combinatorial operators of the $C_\infty$-algebra.

\section{Funding and/or Conflicts of Interests/Competing Interests}

ANH was supported in part by NSF DMS-2208581. DBE was supported in part by  NSF DMS-2205835. Conflict of Interest: The authors declare that they have no conflict of interest.

\bigskip
\textbf{Acknowledgement:} We thank Scott Wilson for discussions and for pointing out the connections to $C_\infty$-algebra and its implications.

\bibliography{wedge}
\bibliographystyle{acmdoi}

\end{document}